\documentclass{amsart}
\usepackage{graphicx} 
\usepackage{amsmath, amssymb}
\usepackage{amsthm}
\usepackage{mathtools}
\usepackage{bm}
\usepackage{tikz-cd}
\usepackage{hyperref}
\usepackage[capitalize,noabbrev]{cleveref}
\usepackage[font=small,labelfont=bf]{caption}

\newcommand{\RR}{\mathbb{R}}

\DeclareMathOperator{\Var}{Var}

\numberwithin{equation}{section}
\newtheorem{thm}[equation]{Theorem}
\newtheorem{cor}[equation]{Corollary}

\theoremstyle{definition}
\newtheorem{defn}[equation]{Definition}
\newtheorem{rem}[equation]{Remark}
\newtheorem{ex}[equation]{Example}



\title{Sharpness characterizes Hill functions}
\author{Marc Stephan}
\address{Center for Systems Biology Dresden, Pfotenhauerstr. 108, 01307 Dresden, Germany}
\address{Max Planck Institute of Molecular Cell Biology and Genetics, Pfotenhauerstr. 108, 01307 Dresden, Germany}
\address{TUD Dresden University of Technology, Institute for Geometry, 01069 Dresden, Germany}
\email{mastepha@mpi-cbg.de}

\date{\today}

\begin{document}

\begin{abstract} While long treated as empirical fits, Hill functions have been postulated to be the universal Hopfield barrier for sharpness of input-output responses by Martinez-Corral, Nam, DePace, and Gunawardena. A Hopfield barrier is a fundamental limit on how well biological systems can process information without expending energy. Their case rested on numerical findings for Hill coefficients $4$ and $6$. We give a precise formulation and proof of this: measuring sharpness by the supremum of the derivative in semi-log scale, any rational function $r(x)=(\alpha_0+\alpha_1 x+ \cdots +\alpha_n x^n)/(\beta_0 + \beta_1 x+ \cdots + \beta_n x^n)$ with real coefficients $0\leq \alpha_i\leq \beta_i$ has sharpness at most $n/4$, with equality if and only if $r$ is a Hill function with Hill coefficient $n$.
\end{abstract}

\maketitle

\section{Introduction}\label{sec:introduction}
Hill functions 
\[H_{K,n}(x) = \frac{x^n}{K^n+x^n}\]
with parameter $K>0$ and positive integer Hill coefficient $n$ are widely used in biology, including biochemistry, pharmacology, and physiology. The articles \cite{weiss_hill_1997} and \cite{santillan_use_2008} survey a range of applications and offer cautionary remarks, while \cite{engel_hundred_2013, gesztelyi_hill_2012} provide historical reviews. For $n=1$, the Hill function has a theoretical basis in the classical Michaelis–Menten input-output response. By contrast, for $n>1$, it has been used primarily as an empirical fit to data since its introduction by A.\ V.\ Hill in 1910 \cite{hill_1910}. Recently, Martinez-Corral et al.~have given a biophysical justification for the use of Hill functions in \cite{martinez-corral_hill_2024}, reducing it to a mathematical question about sharpness of rational functions, for which we provide an answer.

We prove in \cref{cor:hillfunction_barrier} that sharpness of any rational function 
\begin{equation}\label{eq:rationalfunctions}
    r(x)=\frac{\alpha_0+\alpha_1 x+ \cdots +\alpha_n x^n}{\beta_0 + \beta_1 x+ \cdots + \beta_n x^n}
\end{equation}
defined on $\RR_{>0}$ with real coefficients $0\leq \alpha_i\leq \beta_i$, measured by the supremum of the derivative in semi-log scale, is at most $n/4$, with equality if and only if $r$ is a Hill function with Hill coefficient $n$. The connection to biology is as follows.

Inspired by Hopfield's insights in \cite{hopfield_kinetic_1974} on the role of energy expenditure in biosynthetic error correction, Estrada et al.~identified a Hopfield barrier for sharpness in gene regulation in \cite{estrada_information_2016}. The Hopfield barrier is a postulated limit on how well an information processing task can be undertaken at thermodynamic equilibrium. This line of work culminated in \cite{martinez-corral_hill_2024} which argues, on the basis of numerical calculations for $4$ and $6$ input binding sites, that the Hill function with coefficient $n$
is the universal Hopfield barrier for the sharpness of input-output responses if an input ligand binds at $n$ sites. To formulate the problem precisely in terms of rational functions, they apply the linear framework for input-output responses; see \cite{nam_linear_2022,nam_linear_2023}. A small example of a linear framework graph is shown in \cref{fig:intro_combined}A. The response $r(x)$ is a function of a chosen input ligand with concentration $x$, while the concentrations of all other ligands are held constant. Combining the linear framework with the coarse graining from \cite{biddle_allosteric_2021}, the key conclusion for our purposes is
that input-output responses of any Markov process model at thermodynamic equilibrium with $l$ input binding sites can be expressed by a rational function as in \eqref{eq:rationalfunctions} such that $n\leq l$, and with the additional condition that $\beta_i>0$ for all $0\leq i\leq n$. Hill functions of coefficient $n>1$ are not of this form since their intermediate denominator coefficients $\beta_1,\ldots, \beta_{n-1}$ vanish. To include them, we allow $\beta_i=0$ in \cref{cor:hillfunction_barrier}.

Input-output responses of Markov process models at steady states that are not of thermodynamic equilibrium can be described by rational functions as well, but the degrees of the numerator and denominator polynomials can be much larger than the number of binding sites; see \cite[page~6]{martinez-corral_hill_2024}, \cite{owen_size_2023}. In this sense, Hill functions provide a Hopfield barrier in that a steady state input-output response with $l$ binding sites that exceeds the sharpness of the Hill function $H_{K,l}(x)$ requires energy to maintain the steady state. Mathematically, the Hopfield barrier can already be surpassed with just one binding site as shown in \cref{ex:not_thermal_equilibrium}.

Our measure of sharpness\,---\,the maximal steepness in semi-log scale\,---\,differs from that of \cite{martinez-corral_hill_2024} and deserves comment. Semi-log scale is standard for fitting biological data \cite[Chapter~A.1]{motulsky_fitting_2004}, and we use it here with the natural logarithm $\log$ rather than base $10$. This measure is automatically scale invariant.

\subsection*{Outline}
In \cref{sec:input-output}, we recall the sharpness measure from \cite{martinez-corral_hill_2024} and relate it to ours. The characterization of Hill functions given by \cref{cor:hillfunction_barrier} is proved in \cref{sec:hopfield} as a consequence of \cref{thm:inequality_rational_function}. As an application, we consider an allosteric system with two confirmations and one binding site in \cref{sec:allosteric} and show that it can surpass the Hopfield barrier for steady states that are not of thermodynamic equilibrium.

\subsection*{Acknowledgments and AI disclosure} It is my pleasure to thank Heather Harrington, Oskar Henriksson, and Adam Lamson for valuable feedback, and Daniela Egas Santander for helpful conversations and support for plotting \cref{fig:intro_combined}C. A first instance of \cref{cor:hillfunction_barrier} has been developed in conversation with Anthropic Claude Opus 4.6 with the use of semi-log scale derivatives and Popoviciu's inequality suggested by the LLM. Claude Sonnet 4.6 has been used to polish parts written by the author in \cref{sec:introduction,sec:discussion}, and to write Python code to plot \cref{fig:intro_combined} and \cref{fig:random_functions_rprime_x_normalized}. The parameter values in \cref{ex:not_thermal_equilibrium} have been found in conversation with OpenAI GPT-5.4. The author assumes responsibility for all content.

\section{Sharpness of input-output responses}\label{sec:input-output}
Let $n\geq 0$ be a non-negative integer and let $r \colon \RR_{>0} \to \RR$ be a rational function
    \[r(x)=\frac{\alpha_0+\alpha_1 x + \ldots +\alpha_n x^n}{\beta_0 + \beta_1 x + \ldots + \beta_n x^n}\]
with real coefficients $0\leq \alpha_i\leq \beta_i$ for $i=0,\ldots, n$.

\begin{defn}
    The \emph{(absolute) sharpness} of $r$ is the supremum of $\left|\frac{dr}{d\log(x)}\right|$.
\end{defn}
Depending on the context, one may consider the \emph{positive sharpness} (as in the introduction) and the \emph{negative sharpness} defined by the supremum and infimum of $\frac{dr}{d\log(x)}$, respectively.
For $z=\log(x)$ we have 
\[\frac{d}{dz} (r(e^z)) = \frac{dr}{dx}(e^z)\cdot e^z = r'(x)\cdot x\] so that sharpness is scale-invariant. Indeed, any rescaled function
$\tilde{r}(x)=r(cx)$ for $c>0$ satisfies 
\[\sup_{x>0} \left|\tilde{r}'(x) \cdot x \right|= \sup_{x>0} \left|c\cdot r'(cx) \cdot x \right|= \sup_{x>0} \left|x \cdot r'(x)\right|\,.\]
Since $\frac{dr}{d\log x} = x\cdot r'(x)$, sharpness involves both steepness of $r$ and position $x$, but differs from the methodology in \cite{martinez-corral_hill_2024}. There, sharpness is defined by pairs of non-negative real numbers $(p(r),s(r))$ as follows: Consider the normalized function $q(y)=r(x_{0.5}\cdot y)$, where $x_{0.5}$ is the infimum over all $x>0$ for which $r(x)$ is halfway in-between the infimum and supremum of $r$. The largest absolute value of the derivative of $q(y)$ defines the `steepness' $s(r)$, and $p(r)$ denotes its first occurring position.

While the precise relationship between the two measures of sharpness remains to be explored, they can be calculated for Hill functions. 

Henceforth, we write $H_n(x)=\frac{x^n}{1+x^n}$ for the Hill function $H_{K,n}(x)$ with $K=1$

\begin{ex}
Let $H_{K,n}(x)$ be a Hill function with coefficient $n>0$. In this case, $x_{0.5}=K$ so that the normalization is $H_n(y)=\frac{y^n}{1+y^n}$. The corresponding sharpness pair of position and steepness is
\[p(H_{K,n})=\left(\frac{n-1}{n+1}\right)^{1/n},\quad s(H_{K,n})=\frac{(n-1)^{(n-1)/n}(n+1)^{(n+1)/n}}{4n}\,.\]
We will show in the proof of \cref{cor:hillfunction_barrier}, that the sharpness $\sup_x H'_{K,n}(x)\cdot x$ is $n/4$ with maximal value at $x=K$. In particular, the sharpness for the normalized Hill function $H_n(y)$ is achieved at $y=1$. For $n\to \infty$ the position $p(H_{K,n})$ converges to $1$ and the difference 
\[p(H_{K,n})\cdot s(H_{K,n})-\frac{n}{4}= \frac{n^2-1}{4n}-\frac{n}{4}\]
converges to zero. 
\end{ex}

\cref{fig:intro_combined}B.1 displays examples of rational functions $r$ for $n=4$ with corresponding normalizations $q$ in \cref{fig:intro_combined}B.2 and pairs $(p(r),s(r))$ in \cref{fig:intro_combined}C. The sharpness $n/4=1$ of the Hill function $H_4$ bounds the sharpness of $r$. In semi-log scale, normalizing and in fact any rescaling corresponds to a horizontal shift. In particular, the sharpness of $r$ and of $q$ agree as illustrated in \cref{fig:intro_combined}D.1, D.2.

\begin{figure}[!h]
    \centering
      \captionsetup{width=\linewidth}
    \makebox[\textwidth][c]{
        \includegraphics[width=1.2\textwidth]{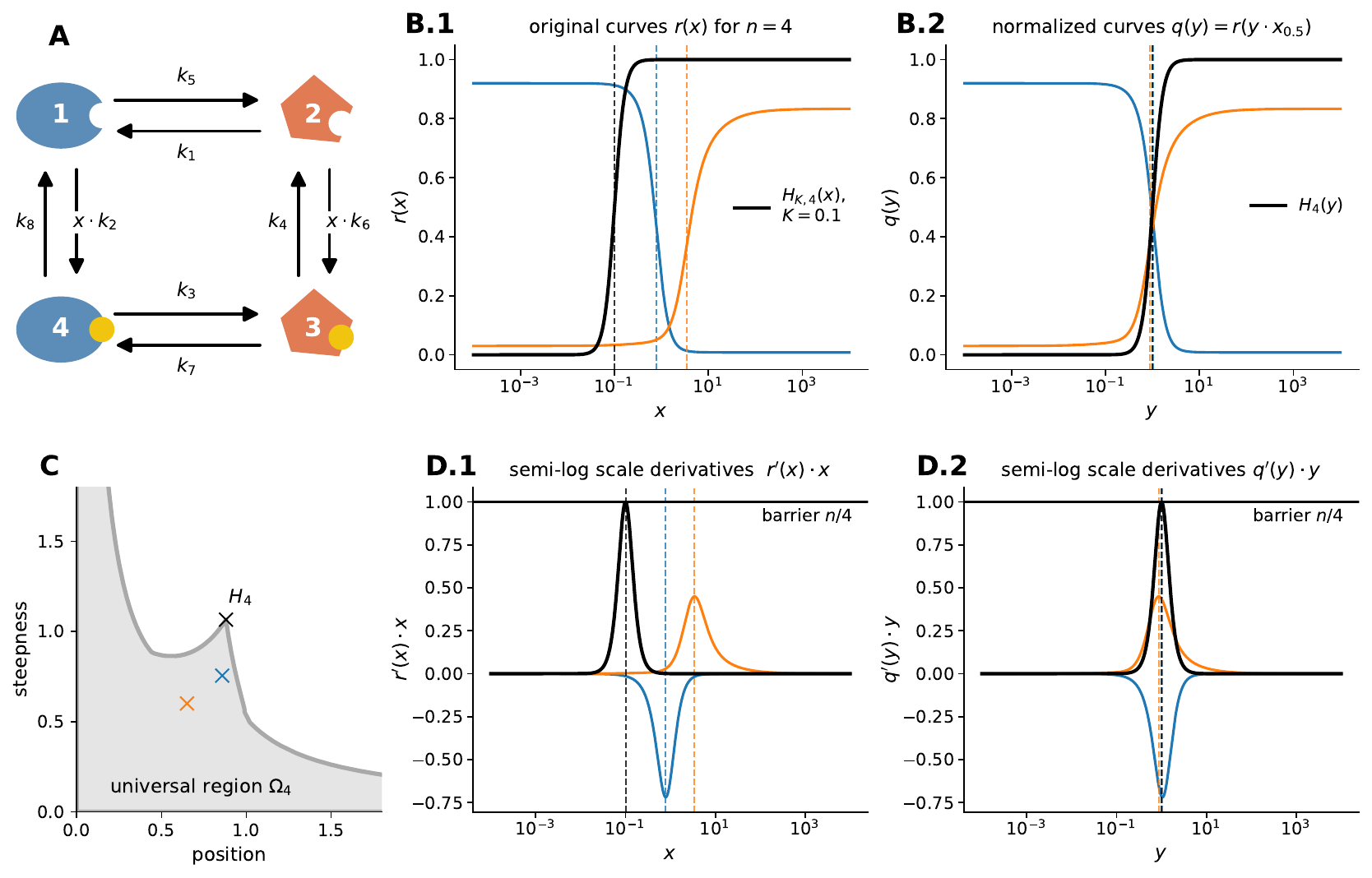}
        }
    \caption{Sharpness and the Hill function $H_4(x)=\frac{x^4}{1+x^4}$ as the universal Hopfield barrier for $n=4$. (A) Linear framework graph representing the binding of a ligand (yellow disk) to a biomolecule with one binding site, but two conformations. (B.1) Two example rational functions $r$ (blue and orange) satisfying the conditions from \cref{cor:hillfunction_barrier} for $n=4$ and the Hill function $H_{K,4}(x)= \frac{x^4}{(1/10)^4 +x^4}$ for $K=1/10$ plotted in semi-log scale. (B.2) Corresponding normalized functions. When plotting in semi-log scale, the normalization $q$ of the rational function $r$ is the horizontal shift satisfying $q(1)=(M-m)/2$ and $q(y)<(M-m)/2$ for $y<1$, where $m=\inf_x r(x)$ and $M=\sup_x r(x)$. (C) Steepness and position for the rational functions $r$ (blue and orange crosses) and the Hill function $H_4$ (black cross) together with the (numerically approximated) universal region $\Omega_4$ of all possible pairs from \cite{martinez-corral_hill_2024}, \cite{martinez-corral_rosamcuniversal-boundaries-hopfield-barrier_2024}. (D.1) The sharpness of $r$ defined as the supremum of $\left|\frac{dr}{d\log(x)}\right|=\left|r'(x)\cdot x\right|$ is bounded by the sharpness $n/4=1$ of the Hill function. The vertical dashed line is at the $x$-value where the maximum of $|r'(x)\cdot x|$ occurs. (D.2) Sharpness for the normalizations $q(y)$. The plot of $q'(y)\cdot y$ is a horizontal shift of the corresponding plot for $r'(x)\cdot x$ from (D.1).}
    \label{fig:intro_combined}
\end{figure}

\section{Hill functions as the universal Hopfield barrier}\label{sec:hopfield}
We will establish that Hill functions are the universal Hopfield barrier for sharpness in \cref{cor:hillfunction_barrier} as a consequence of the following result in which we consider more general coefficients.
\begin{thm}\label{thm:inequality_rational_function} Let $n\geq 0$ be a non-negative integer and let $r \colon \RR_{>0} \to \RR$ be a rational function
    \[r(x)=\frac{\alpha_0+\alpha_1 x + \ldots +\alpha_n x^n}{\beta_0 + \beta_1 x + \ldots + \beta_n x^n}\]
with real coefficients $\alpha_i$ and nonnegative coefficients $\beta_i\geq 0$ such that $\alpha_i=0$ whenever $\beta_i=0$. Let 
\begin{align*}
    m_{\alpha/\beta} = \min_{i} \alpha_i/\beta_i \quad \text{and}\quad M_{\alpha/\beta} =\max_{i} \alpha_i/\beta_i
\end{align*}
where the minimum and maximum are taken over all $i$ such that $\beta_i\neq 0$.

Then the semi-log scale derivative $\frac{d}{d \log(x)} r = x \cdot r'(x)$ is bounded by
\[\left|x\cdot r'(x) \right| \leq 
\frac{(M_{\alpha/\beta}-m_{\alpha/\beta})n}{4}\]
for any $x>0$. The bound is obtained for some $x$ if and only if $r(x)$ is of the form
\[r(x)=\frac{\alpha_0 + \alpha_n x^n}{\beta_0 + \beta_n x^n}\,.\]
\end{thm}

\begin{proof}
It is straightforward to check the claims for constant rational functions. From now on, suppose that $r(x)$ is not constant, and hence $m_{\alpha/\beta}<M_{\alpha/\beta}$.

Let $I$ be the set of $k\in\{0,\ldots, n\}$ with $\beta_k\neq 0$. Thus $D(x)=\sum_{k\in I} \beta_k x^k$ is the denominator of $r(x)$ and the numerator is $N(x)=\sum_{k\in I} \alpha_k x^k$ since $\beta_k=0$ implies $\alpha_k=0$ for any $k$. 

For $k\in I$ we write $w_k=\alpha_k/\beta_k$ and $p_k(x)=\beta_k x^k/D(x)$.  The definition of $p_k(x)$ is chosen so that $\sum_{k\in I} p_k = 1$.

We obtain
\begin{align*}
        x \cdot r' &= x \frac{N'}{D} - x\frac{D'N}{D^2} \\
            &= \sum_{k\in I} k w_k p_k  - r \sum_{k\in I} k p_k \\
            &= \sum_{k\in I} (w_k -r) k p_k\,.
\end{align*}
Let $\overline{k} =\sum_k  k p_k$. Using that $\sum_k (w_k-r)p_k=0$ and the Cauchy-Schwarz inequality, we further obtain
\begin{align*}
    \left| x \cdot r' \right| & = \left|\sum_{k} (w_k -r) k p_k \right|\\
 &= \left|\sum_k  (w_k -r) k p_k - \overline{k} \sum_k (w_k -r) p_k \right| \\
 &= \left| \sum_k (w_k-r)(k-\overline{k}) p_k \right| \\
 &\leq \sqrt{\sum_k (w_k-r)^2p_k} \sqrt{\sum_k (k-\overline{k})^2p_k} \,.
\end{align*}
We will apply Popoviciu's inequality on variances: For real numbers $a<b$ and any real-valued random variable $X$ bounded by $P(a\leq X\leq b)=1$, the variance is bounded by
\[\Var(X)=E\left[(X-E(X))^2\right]\leq (b-a)^2/4\]
with equality if and only if $P(X=a)=1/2=P(X=b)$; see \cite{popoviciu_1935,muilwijk_1966,bhatia_davis_2000}. For any $x>0$ we consider $I$ together with the function $k\mapsto p_k(x)$ as a finite probability space.  The functions $X(k)=w_k$ and $Y(k)=k$ are bounded random variables on $I$. The random variable $X$ is bounded by $m_{\alpha/\beta}\leq X(k)\leq M_{\alpha/\beta}$ by definition of $m_{\alpha/\beta}$ and $M_{\alpha/\beta}$, and has expected value $\sum_k w_k p_k(x)= r(x)$. The random variable $Y$ is bounded by $0\leq Y(k)\leq n$ and has expected value $\overline{k}$ by definition.
Now, Popoviciu's inequality applied twice yields
\begin{align*}
    \left|x\cdot r'(x) \right| \leq \sqrt{(M_{\alpha/\beta}-m_{\alpha/\beta})^2/4} \cdot \sqrt{n^2/4} = \frac{(M_{\alpha/\beta}-m_{\alpha/\beta})n}{4}
\end{align*}
as desired.

If $\left|x_0\cdot r'(x_0)\right|=\frac{(M_{\alpha/\beta}-m_{\alpha/\beta})n}{4}$ for some $x_0$,
then we need to have equality in Popoviciu's inequality for $X$ and $Y$. In particular, $w_k\in \{m_{\alpha/\beta},M_{\alpha/\beta}\}$ and $k\in \{0,n\}$ for any $k\in I$. The latter condition implies that $\beta_k=0$ for $k=1,\ldots, n-1$. Thus by assumption we have $\alpha_k=0$  for $k=1,\ldots,n-1$ and $r$ is of the form 
\[r(x)=\frac{\alpha_0+\alpha_n x^n}{\beta_0+ \beta_n x^n}\,.\] 
Any such non-constant $r$ satisfies $\beta_0\neq 0$ and $\beta_n\neq 0$. Finally, the function
\[\left|x \cdot r' \right|= \left|n\frac{(\alpha_n\beta_0-\alpha_0\beta_n)x^n}{(\beta_0+\beta_nx^n)^2}\right|
\]
reaches the desired upper bound at $x=(\beta_0/\beta_n)^{1/n}$.
\end{proof}

\begin{rem}\label{rem:bhatia-davis}
    Popoviciu's inequality for a random variable $X$ bounded by $P(a\leq X\leq b)=1$ follows from the Bhatia--Davis inequality, i.e., that 
    \[\Var(X)\leq (b-E(X))(E(X)-a).\]
    In the situation of \cref{thm:inequality_rational_function}, the Bhatia--Davis inequality provides the bound
    \[\left|x\cdot r'(x)\right| \leq \sqrt{(M_{\alpha/\beta}- r(x))(r(x)-m_{\alpha/\beta})}\cdot \frac{n}{2}.\]
\end{rem}

In addition to the positive Hill functions $H_n(x)=\frac{x^n}{K^n+x^n}$ to describe activator-driven expression, negative Hill functions $\frac{K^n}{K^n + x^n}=\frac{x^{-n}}{K^{-n}+x^{-n}}$ with $K>0$ are used to describe repression.

\begin{cor}\label{cor:hillfunction_barrier}
    Let $n>0$ be  a positive integer and let $r \colon \RR_{>0} \to \RR$ be a rational function
    \[r(x)=\frac{\alpha_0+\alpha_1 x + \ldots +\alpha_n x^n}{\beta_0 + \beta_1 x + \ldots + \beta_n x^n}\]
with real coefficients $0\leq \alpha_i\leq \beta_i$ for $i=0,\ldots,n$. Then the semi-log scale derivative $\frac{d}{d \log(x)} r = x \cdot r'(x)$ is bounded by
\[-\frac{n}{4}\leq x\cdot r'(x) \leq 
\frac{n}{4}\]
for any $x>0$. The upper bound is obtained for some $x$ if and only if $r(x)$ is a Hill function. The lower bound is obtained for some $x$ if and only if $r(x)$ is a negative Hill function.
\end{cor}
\begin{proof}
Since $0\leq \alpha_i\leq \beta_i$ the conditions to apply \cref{thm:inequality_rational_function} hold, and $0\leq \alpha_i/\beta_i\leq 1$ for any $i$ with $\beta_i\neq 0$. Thus we have $0\leq m_{\alpha/\beta}$ for the minimum $m_{\alpha/\beta}$ in \cref{thm:inequality_rational_function} and $M_{\alpha/\beta}\leq 1$ for the maximum $M_{\alpha/\beta}$ in \cref{thm:inequality_rational_function}. The bound in \cref{thm:inequality_rational_function} yields 
\[\left|x\cdot r'(x)\right| \leq \frac{M_{\alpha/\beta}-m_{\alpha/\beta}}{4}n\leq \frac{n}{4}\,.\]

If the bound $n/4$ is reached, then 
\[r(x)=\frac{\alpha_0 + \alpha_n x^n}{\beta_0 + \beta_n x^n}\]
and $m_{\alpha/\beta}=0$, $M_{\alpha/\beta}=1$ so that either $\alpha_0/\beta_0=0$ and $\alpha_n/\beta_n=1$ or $\alpha_0/\beta_0=1$ and $\alpha_n/\beta_n=0$.
The first case implies $\alpha_0=0$ and $\alpha_n=\beta_n$ so that $r(x)= \alpha_n x^n/(\beta_0+\alpha_n  x^n)=x^n/(\beta_0/\alpha_n +  x^n)$. If $\beta_0=0$, then the bound is not reached. Otherwise $r$ is a Hill function with $K=(\beta_0/\beta_n)^{1/n}$. For any Hill function $H_{K,n}(x)=\frac{x^n}{K^n+x^n}$, we have 
\[
x\cdot  H_{K,n}'(x) = \frac{nK^n  x^n}{(K^n+x^n)^2} 
\]
which reaches the maximum value $n/4$ at $x=K$.
The second case implies $\alpha_0=\beta_0$ and $\alpha_n=0$ so that $r=\frac{\alpha_0}{\alpha_0+ \beta_n x^n}$. If $\alpha_0=0$, then the bound is not reached. Otherwise, $r=\frac{K^n}{K^n+x^n}$ is a negative Hill function with $K=(\beta_0/\beta_n)^{1/n}$. For any negative Hill function, we obtain $x\cdot r'=\frac{-n K^n x^n}{(K^n+x^n)^2}$ for which the lower bound $-n/4$ is reached at $x=K$.
\end{proof}

For $r(x)$ as in \cref{cor:hillfunction_barrier}, Owen and Horowitz considered the elasticity
\[\left|\frac{d\log r(x)}{d\log x}\right|=\left| \frac{r'(x)\cdot x}{r(x)}\right|\]
and bounded it by $n$ in \cite{owen_size_2023}. Their bound for steepness in log-log scale compares to bounds for steepness in semi-log scale as follows.
\begin{rem} If we use the Bhatia--Davis inequality as in \cref{rem:bhatia-davis}, we obtain
\[\left|r'(x)\cdot x\right| \leq \sqrt{r(x)\cdot(1-r(x))}\cdot\frac{n}{2}\,.\]
This bound is stronger than the bound 
\[\left|r'(x)\cdot x\right| \leq r(x)\cdot n\]
from \cite[(32)]{owen_size_2023} if $r(x)> 1/5$, but weaker if $r(x)<1/5$. The function $\sqrt{t(1-t)}$ for $0\leq t\leq t$ has maximum value $1/2$ at $t=1/2$, explaining the significance of the function value $1/2$ for $r(x)$.
\end{rem}

The condition $\alpha_i\leq \beta_i$ is necessary in \cref{cor:hillfunction_barrier}.
\begin{ex}
     For $r(x)= (\alpha_1 x + x^2)/(1+x^2)$ the supremum over $x>0$ of the semi-log scale derivative 
     \[x\cdot r'= \frac{-\alpha_1 x^3 + 2x^2 + \alpha_1 x}{(x^2 +1)^2}
     \]
     is bounded below by its value at $x=1/2$ and thus tends to infinity as $\alpha_1\to \infty$.
\end{ex}

\begin{rem}
    If $r(x)$ in \cref{cor:hillfunction_barrier} is not constant, then we may consider the function normalized in the vertical direction $u(x)=\frac{r(x)-m_{\alpha/\beta}}{M_{\alpha/\beta}-m_{\alpha/\beta}}$
    and the stronger bound from \cref{thm:inequality_rational_function} for the function $r(x)$ is just \cref{cor:hillfunction_barrier} applied to the rational function $u(x)$.
\end{rem}

For numerical examples in \cite{martinez-corral_hill_2024}, coefficients $\beta_i$ and $\alpha_i$ are sampled as follows. Fix a real number $a>0$. Then choose $\beta_i$ with $\log_{10}(\beta_i)$ uniformly at random in the closed interval $[-a,a]$ and $\alpha_i$ with $\log_{10}(\alpha_i)$ uniformly at random in $[-a, \log_{10}(\beta_i)]$. \cref{fig:random_functions_rprime_x_normalized} provides three such examples of rational functions $r$ with $a=4$ and numerator and denominator polynomials of degree $n=8$ together with the corresponding normalized functions $u(x)$.
The Hill function $H_8(x)=\frac{x^8}{1+x^8}$ with sharpness $n/4=2$ is the universal Hopfield barrier.

\begin{figure}[!h]
    \centering
    \captionsetup{width=\linewidth}
    \makebox[\textwidth][c]{
        \includegraphics[width=1.3\textwidth]{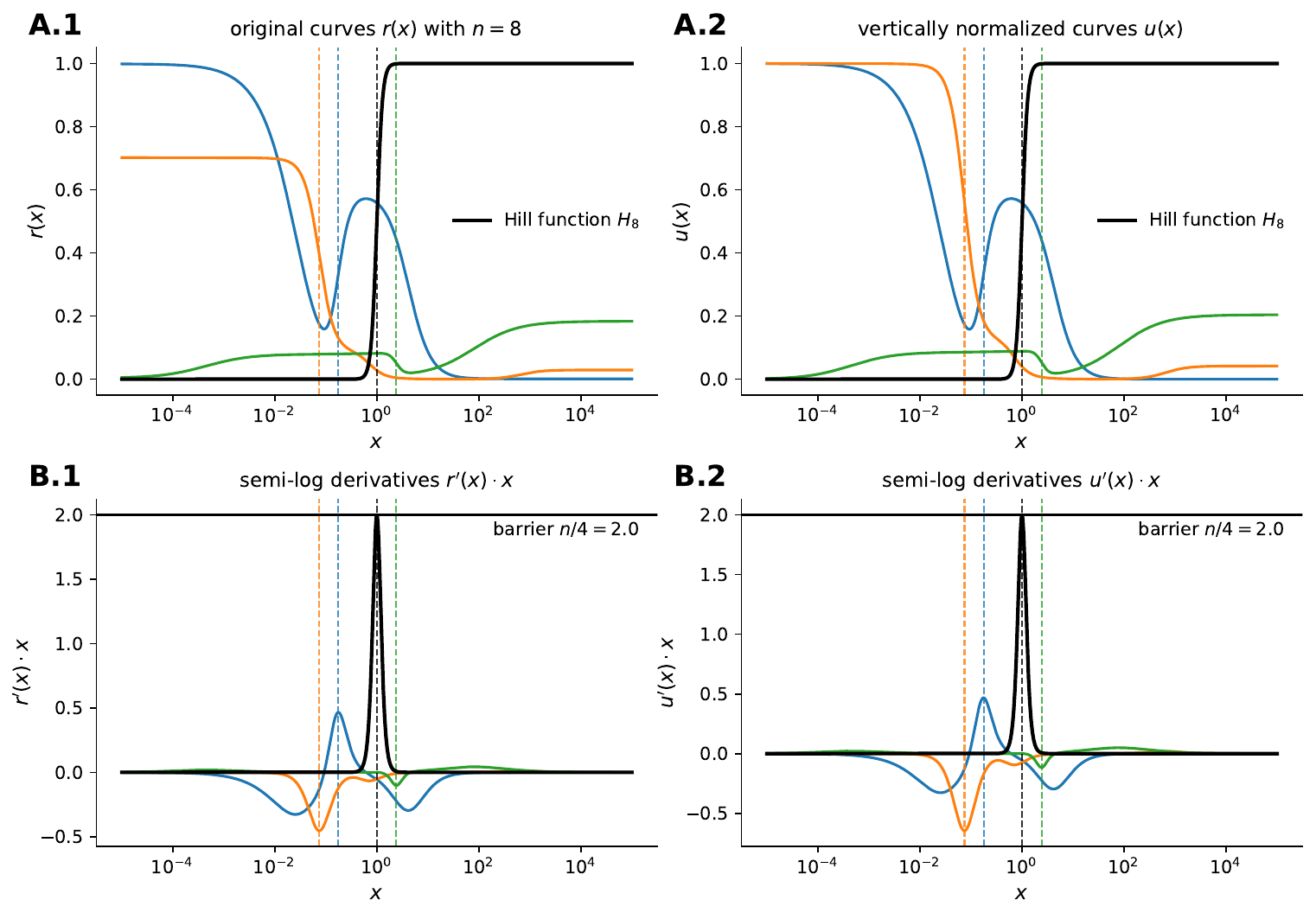}
        }
    \caption{Sharpness of random input-output responses. (A.1) Three random input-output responses $r(x)=(\alpha_0+\ldots+\alpha_n x^n)/(\beta_0+\ldots +\beta_n x^n)$ for $n=8$ and the Hill function $H_8(x)=x^8/(1+x^8)$. (A.2) Normalizations in the vertical direction $u(x)=(r(x)-m_{\alpha/\beta})/(M_{\alpha/\beta}-m_{\alpha/\beta})$  with respect to $M_{\alpha/\beta}=\max_i \alpha_i/\beta_i$ and $m_{\alpha/\beta} =\min_i\alpha_i/\beta_i$. (B.1) and (B.2) The sharpness $\sup_x |r'(x)\cdot x|$ of $r(x)$ and of its vertical normalization is bounded by the sharpness of the Hill function $n/4=2$. The vertical dashed lines indicate the $x$-value, where the sharpness is attained.}
    \label{fig:random_functions_rprime_x_normalized}
\end{figure}


\begin{rem}
    The steepness with input in $\log_{10}$-scale is 
    \[\frac{d}{d\log_{10}(x)}r=\log(10)\cdot r'(x)\cdot x\,.\]
\end{rem}

\section{Application: allosteric system with two conformations}\label{sec:allosteric}
We consider the allosteric system with two conformations and a single ligand binding site treated in \cite[Section~4, Figure~3b]{nam_linear_2022} and depicted in \cref{fig:intro_combined}A for parameters $k_1,k_2,\ldots, k_8$. The steady-state probability $r(x)$ of the ligand $x$ being bound depends on whether the steady state is of thermodynamic equilibrium. It has been computed in \cite{nam_linear_2022} using the linear framework. At thermodynamic equilibrium, the resulting input-output function is of Michaelis-Menten type
\[r_{eq}(x)=\frac{Ax}{B+Ax}\]
with $A=\frac{k_2}{k_8}(1 + \frac{k_3}{k_7})$ and $B=1+\frac{k_5}{k_1}$. If the steady state is not of thermodynamic equilibrium, then we obtain the input-output function
\[r_{ne}(x)=\frac{A^* x + B^* x^2}{C^* + D^* x + B^* x^2}\]
with 
\begin{align*}
A^* &= k_1 k_2 k_3 + k_3 k_5 k_6 + k_5 k_6 k_8 + k_1 k_2 k_7 + k_5 k_6 k_7 + k_1 k_2 k_4\\
B^* &= k_2k_3k_6 + k_2 k_6 k_7 \\
C^* &= k_1 k_7 k_8 + k_1 k_3 k_4 + k_1 k_4 k_8 + k_5 k_7 k_8 + k_3 k_4 k_5 + k_4 k_5 k_8 \\
D^* &= A^* + k_2 k_3 k_4 + k_6 k_7 k_8\,.
\end{align*}

\begin{ex}\label{ex:not_thermal_equilibrium}
   The function $r_{ne}$ for \[k_1=1,\quad k_2=2,\quad k_3=1,\quad k_4=1,\quad
k_5=10,\quad k_6=1,\quad k_7=1,\quad k_8=1\]
is
\[
r_{ne}(x)=\frac{36 x + 4 x^2}{33 + 39 x + 4x^2}
\]
and 
\[x\cdot r'_{ne}(x)= 12 \frac{x^3 + 22 x^2 + 99 x}{(4x^2 + 39 x + 33)^2}
\]
achieves the value $(12\cdot 122)/76^2\simeq 0.253$ at $x=1$ and thus surpasses the Hopfield barrier of $1/4$ for a single ligand binding site. 
\end{ex}

\section{Discussion}\label{sec:discussion}
For over a century, Hill functions have been used as empirical fits to biological input-output responses. Martinez-Corral et al.\ \cite{martinez-corral_hill_2024} proposed, on the basis of numerical evidence for $n=4$ and $6$, that Hill functions are the universal Hopfield barrier for sharpness of input-output responses, reducing the question to one about sharpness of rational functions $r$. Their measure of sharpness consists of two values, position and steepness. Using the supremum of $r'(x)\cdot x$ as our measure, we solve the remaining mathematical problem, characterizing Hill functions as the unique maximizers of sharpness. A related bound, using a different sensitivity measure and holding both at and away from thermodynamic equilibrium, was established by Owen and Horowitz \cite{owen_size_2023}. It would be interesting to find exact values of maximal sharpness for specific models, such as the allosteric system with two conformations from \cref{sec:allosteric}, and more broadly to characterize Hopfield barriers for other forms of biological information processing tasks as suggested in \cite{martinez-corral_hill_2024}.

\bibliographystyle{amsalpha}
\bibliography{bibl}

\end{document}